\documentclass[a4paper,12pt]{article}
\usepackage{amsfonts,amssymb,epsfig,amstext,amsthm}
\usepackage[latin1]{inputenc}
\usepackage[english]{babel}

\newtheorem{thm}{Theorem}[section]
\newtheorem{pro}{Proposition}[section]
\newtheorem{cor}{Corollary}[section]
\newtheorem{lem}{Lemma}[section]
\newtheorem{rmk}{Remark}[section]

\begin{document}
\title{\vspace{-0.5in}\parbox{\linewidth}{\footnotesize\noindent }
%
Dualistic Structures on Twisted Product Manifolds
}
\date{}
\maketitle
\begin{center}
\author{Abdoul Salam DIALLO\footnote{abdoulsalam.diallo@uadb.edu.sn}}\\
Universit\'e Alioune Diop de Bambey,\\
UFR SATIC, D\'epartement de Math\'ematiques, \\
B.P. 30, Bambey, S\'en\'egal \\
and\\
Leonard Todjihound\'e\footnote{leonardt@imsp-uac.org},\\
Institut de Math\'ematiques et de Sciences Physiques,\\
B.P. 630, Porto-Novo, B\'enin
\end{center}


\begin{abstract}
In this paper, we show that the projection of a dualistic structure
defined on a twisted product manifold induces dualistic structures
on the base and the fiber manifolds, and conversely. Then under some 
conditions on the Ricci curvature and the Weyl conformal tensor we
characterize dually flat structures on twisted product manifolds.
\end{abstract}

\section{Introduction}
\noindent
Dualistic structures are a fundamental mathematics concept of information 
geometry, specially in the investigation of the natural differential 
geometric structure possessed by families of probability distributions. 
Information geometry is a branch of the mathematics that applied the
technique of differential geometry to the field of probability theory.
This is done by taking probability distributions for a statistical model
as the points of a Riemannian manifold, forming a statistical manifold.
The fisher information metric provides the Riemannian matric (see 
~\cite{A},~\cite{AN} for more information).\\

\noindent
The information geometry is nowdays applied in a broad variety of different 
fields and contexts which include, for instance, information theory, stochastic
processes, dynamical systems and times series, statistical physics, quantum 
systems and the mathematical theory of neural networks ~\cite{AT}.\\

\noindent
Dually flat manifolds constitutes fundamental objets of information
geometry. However, due to the fact that the global theory of dually flat
manifolds is still far from being complete, its range of application
still suffers certain limitations since often only matters of mainly
a local nature can be successfully pursued. Consequently, there is a
strong need and desire for a further understanding of the global
characteristics of dually flat manifolds (see ~\cite{AN},~\cite{AT}).\\

\noindent
In ~\cite{T}, the author obtained that a warped product manifold is 
dually flat if and only if the base manifold is dually flat 
and the fiber manifold is a constant sectional curvature. This result 
is closed related to the fact that the warping function is defined
only on the base manifold and do not depend on the points on the
fiber manifold. In the present paper, we investigate dually structures
on twisted product manifolds and under some conditions we characterize
dually flat twisted product manifolds.

\section{Preliminaries}

\subsection{Dualistic structures}
\noindent
Let $(M,g)$ be a Riemannian manifold and $\nabla$ an affine connection on $M$. A 
connection $\nabla^*$ is called \textit{conjugate connection} (or \textit{dual 
connection}) of $\nabla$ with respect to the metric $g$ if
\begin{eqnarray}\label{e1}
 X\cdot g(Y,Z) = g(\nabla_X Y,Z) + g(Y,\nabla^{*}_{X} Z), 
\end{eqnarray}
for arbitrary $X,Y,Z \in \Gamma(TM)$. The triple $(g,\nabla,\nabla^*)$
satisfying (\ref{e1}) is called \textit{dualistic structure} on $M$.\\

\noindent
The geometry of conjugate connections is a natural generalization of 
geometry of Levi-Civita connections from Riemannian manifolds theory. Conjugate 
connections arise from affine differential geometry and from geometric theory of 
statistical inferences \cite{AN}. In ~\cite{T}, the author 
proved that the projection of a dualistic structure defined on a warped 
product space induces dualistic structures on the base and the fiber 
manifold. Recently in ~\cite{D}, the author extended the construction of 
doubly warped product for geometry of conjugate connections .\\

\begin{pro}
The torsion tensors $T^{\nabla}$ and $T^{\nabla^*}$ of $\nabla$ and $\nabla^*$,
respectively, satisfy: 
\begin{eqnarray*}
g(T^{\nabla}(X,Y),Z) = g(T^{\nabla^*}(X,Y),Z) + (\nabla^* g)(X,Y,Z) - (\nabla^* g)(Y,X,Z)
\end{eqnarray*}
for any $X,Y, Z \in \Gamma(TM)$.
\end{pro}

\begin{proof}
From the torsion tensor equation, we have:
\begin{eqnarray*}
g(T^{\nabla}(X,Y),Z) &:=& g(\nabla_X Y,Z) - g(\nabla_Y X,Z) -g([X,Y],Z),\\
&=&X\cdot g(Y,Z) - g(Y,\nabla^{*}_{X} Z) - Y\cdot g(X,Z)\\
&+& g(X,\nabla^{*}_{Y} Z) -g(\nabla^{*}_{X} Y - \nabla^{*}_{Y} X - T^{\nabla^*}(X,Y),Z)\\
&:=& g(T^{\nabla^*}(X,Y),Z) +(\nabla^* g)(X,Y,Z) - (\nabla^* g)(Y,X,Z)
\end{eqnarray*}
for any $X,Y,Z \in \Gamma(TM)$.
\end{proof}

\begin{cor}
If $\nabla^* g = 0$ is, then $T^{\nabla} = T^{\nabla^*}$.
\end{cor}

\noindent
Let $C(X,Y,Z) = \nabla_X g(Y,Z)$ the cubic form of $(\nabla,g)$ and 
$C^* (X,Y,Z) = \nabla^{*}_{X} g(Y,Z)$ the cubic form of $(\nabla^*,g)$. We have the
following property:
\begin{pro}
The cubic form of $(\nabla,g)$ is symmetric if and only the cubic form of 
$(\nabla^*,g)$ is symmetric.
\end{pro}

\begin{proof}
From definition, we have:
\begin{eqnarray*}
(\nabla^* g)(X,Y,Z) &:=& X\cdot g(Y,Z) - g(\nabla^{*}_{X} Y,Z) - g(Y,\nabla^{*}_{X} Z)\\
&=& X\cdot g(Y,Z) - X\cdot g(Y,Z) + g(Y,\nabla_X Z) \\
&-& X\cdot g(Y,Z) + g(Z,\nabla_X Y)\\
&:=& - (\nabla g)(X,Y,Z).
\end{eqnarray*}
for any $X,Y,Z \in \Gamma(TM)$.
\end{proof}

\begin{cor}
If $\nabla^* g$ is symmetric and $\nabla^*$ is torsion free, then $\nabla g$ is symmetric
and $\nabla$ is torsion free too.
\end{cor}

\noindent
The triple $(M,\nabla,g)$ is called \textit{statistical manifold} 
if $\nabla$ is a torsion free affine connection and its cubic form is 
symmetric. If $\nabla^*$ is conjugate connection with respect to $g$ 
on $M$, then $(M,\nabla^*,g)$ is also statistical manifold called 
the \textit{dual statistical manifold} of $(M,\nabla,g)$. The statistical
manifold was introduced by S. Amari ~\cite{A}, it connects information
geometry, affine differential geometry and Hessian geometry. \\

\noindent
Let $R$ and $R^*$ the curvature tensors of $\nabla$ and $\nabla^*$ respectively. We have 
also the following:

\begin{pro}
The curvature tensors $R$ and $R^*$ of $\nabla$ and $\nabla^*$ are related by
\begin{eqnarray*}
g(R(X,Y)Z,W) = -g(R^*(X,Y)W,Z) 
\end{eqnarray*}
for any $X,Y,Z,W \in \Gamma(TM)$.
\end{pro}

\begin{cor}
$R =0$ if and only if $R^* = 0$.
\end{cor}

\noindent
The manifold $M$ endowed with a dualistic structure $(g,\nabla,\nabla^*)$ 
is called a \textit{dually flat space} if both dual connections $\nabla$ 
and $\nabla^*$ are torsion free and flat; that is the curvature tensors 
with respect to $\nabla$ and $\nabla^*$ respectively vanishe identically. 
This does not imply that the manifold is Euclidean, because the Riemannian 
curvature due to the Levi-Civita connections does not necessarily vanish. 
Moreover the existence of a dually flat structure on a manifold points 
out some topological and geometrical properties of the manifold. For example 
if a manifold $M$ admits a dually flat structure $(g,\nabla,\nabla^*)$ and 
if one of the dual connection, say $\nabla$, is complete, then only the 
first homotopy group of $M$ is non trivial, and any two points in $M$ can 
be joined by a $\nabla-$geodesic ~\cite{AT}.\\

\subsection{Twisted product manifolds}

\noindent
Let $(B,g_B)$ and $(F,g_F)$ be Riemannian manifolds of dimensions $r$
and $s$ respectively, and let $\pi:B \times F \rightarrow B$ 
and $\sigma: B \times F \rightarrow F$ be the canonical projections. Also let 
$b:B\times F \rightarrow (0,\infty)$ be positive smooth function. Then the 
\textit{twisted product} of Riemannian manifolds $(B,g_B)$ and $(F,g_F)$ with 
twisting function $b$ is the product manifold $B \times F$ 
with metric tensor 
\begin{eqnarray*}
g = g_B \oplus b^2 g_F
\end{eqnarray*}
given by
\begin{eqnarray*}
g= \pi^* g_B + (b\circ \pi)^2 \sigma^* g_F.
\end{eqnarray*}
We denote this Riemannian manifold $(M,g)$ by $B\times_b F$. In particular, if 
$b$ is constant on $F$, then $B\times_b F$ is called the \textit{warped product} 
of $(B,g_B)$ and $(F,g_F)$ with warping function $b$. Moreover if $b= 1$, then 
we obtain a \textit{direct product}. If $b$ is not constant, then we have a 
\textit{proper twisted product}.\\

\noindent
Let $\mathcal{L} (B)$ (respectivement $\mathcal{L}(F)$) be the set of all vector 
fields on $B\times F$ which is the horizontal lift (respectivement the vertical 
lift) of a vector field on $B$ (respectivement on $F$). Thus a vector field on 
$B\times F$ can be written as
\begin{eqnarray*}
A = X + U, \quad \mbox{with} \quad X \in \mathcal{L}(B) \quad \mbox{and}\quad U \in \mathcal{L}(F). 
\end{eqnarray*}
Obviously
\begin{eqnarray*}
\pi_* (\mathcal{L}(B)) =  \Gamma(TB) \quad \mbox{and} \quad \sigma_* (\mathcal{L}(F)) = \Gamma(TF).
\end{eqnarray*}

\noindent
For any vector field $X\in \mathcal{L}(B)$, we denote $\pi_* (X)$ by $\bar{X}$ 
and for any vector field $U\in \mathcal{L}(F)$, we denote $\sigma_* (U)$ by $\bar{U}$. 

\begin{lem}~\cite{BO}
Let $\bar{X}, \bar{Y}, \bar{Z} \in \Gamma(TB)$ and $X,Y,Z \in \mathcal{L}(B)$ be 
their corres\-ponding horizontal lifts. We have: 
\begin{eqnarray}\label{e2}
\bar{X}\cdot g(\bar{Y},\bar{Z})\circ \pi = X\cdot g(Y,Z).
\end{eqnarray}
Also, let $\bar{U}, \bar{V}, \bar{W} \in \Gamma(TF)$ and $U,V,W \in \mathcal{L}(F)$ be 
their corresponding vertical lifts. Then
\begin{eqnarray}\label{e3}
\bar{U}\cdot g(\bar{V},\bar{W})\circ \sigma = U\cdot g(V,W).
\end{eqnarray}
\end{lem}

\noindent
Let $(B,g_B)$ and $(F,g_F)$ be Riemannian manifolds with Levi-Civita 
connection ${}^B\nabla$ and ${}^F\nabla$, respectively, and let $\nabla$ 
denote the Levi-Civita connection and the gradient of the twisted 
product manifold $(B\times_b F)$ of $(B,g_B)$ and $(F,g_F)$ with twisting 
function $b$. We have the following proposition.

\begin{pro}[~\cite{FGKU}]
Let $M=B\times_b F$ be a twisted product manifold with the $g=g_B \oplus b^2 g_F$ and 
let $X,Y \in \mathcal{L}(B)$ and $U,V \in \mathcal{L}(F)$. Then we have
\begin{eqnarray*}
\nabla_X Y &=& {}^B\nabla_{X} Y;\\
\nabla_X U &=& \nabla_U X = X(k)U;\\
\nabla_U V &=& {}^F\nabla_{U} V + U(k)V + V(k)U
-g_F(U,V)\nabla k 
\end{eqnarray*}
where $k=\log b$.
\end{pro}

\noindent
Let $M$ be an $m$-dimensional manifold with the metric tensor $g$. If 
$(E_1,\cdots,E_m)$ is a orthonormal base of $M$, then we define the 
curvature tensor, Ricci curvature and scalar curvature, respectively, 
as follows:
\begin{eqnarray*}
R(X,Y)Z &=& \nabla_X \nabla_Y Z -\nabla_Y \nabla_X Z - \nabla_{[X,Y]} Z;\\
Ricc(X,Y) &=& \sum_{i=1}^{m}g(R(E_i,X)Y,E_i)\\
S &=& \sum_{i} Ric(E_i,E_i).
\end{eqnarray*}
The Weyl conformal curvature tensor field of $M$ is the tensor field 
$C$ of type $(1,3)$ defined by
\begin{eqnarray*}
C(X,Y)Z &=& R(X,Y)Z\\
&+& \frac{1}{m-2}[Ric(X,Z)Y-R(Y,Z)X + g(X,Z)QY-g(Y,Z)QX]\\
&-& \frac{S}{(m-1)(m-2)}[g(X,Z)Y-g(Y,Z)X]
\end{eqnarray*}
for any vector fields $X,Y$ and $Z$ on $M$, where $S$ is scalar curvature.\\

\noindent
Define $h^{k}_{B}(X,Y) = XY(k) - {}^B\nabla_X Y(k)$ for $X,Y \in \mathcal{L}(B)$. If
$X,Y\in \mathcal{L}(B)$ and $V\in \mathcal{L}(F)$, then $XV(k)=VX(k)$ 
and the Hessian form $h^k$ of $k$ on $B\times_b F$ satisfies
\begin{eqnarray*}
h^k (X,V) &=& XV(k)-X(k)V(k),\\
h^k (X,Y) &=& h^{k}_{B}(X,Y).
\end{eqnarray*}

\noindent
Let $R^B$ and $R^F$ be the curvature tensors of $(B,g_B)$ and 
$(F,g_F)$, respectively, and let $R$ be the curvature tensor of 
$B\times_b F$. Then we have the following proposition:

\begin{pro}[~\cite{FGKU}]
Let $M=B\times_b F$ be a twisted product manifold. If $X,Y,Z \in \mathcal{L}(B)$ and
$U,V,W \in \mathcal{L}(F)$, then we have
\begin{eqnarray*}
R(X,Y)Z &=& R^B(X,Y)Z;\\
R(X,Y)U &=& 0;\\
R(X,U)Y &=& \frac{h^{b}_{B}(X,Y)}{f}U;\\
R(U,V)X &=& UX(k)V -VX(k) U;\\
R(X,U)V &=& [X(k)V(k)+h^k(X,V)]U -g(U,V) [X(k)\nabla k + H^k(X)];\\
R(U,V)W &=& R^F (U,V)W + g(U,W)grad_B (V(\log b))
-g(V,U)grad_B(U(\log b))\\ 
&-& \frac{\vert grad_B b\vert^2}{b^2}[g(V,W)U - g(U,W)V].
\end{eqnarray*}
\end{pro}

\begin{pro}[~\cite{FGKU}]
Let $M=B\times_b F$ be a twisted product manifold. If $X,Y \in \mathcal{L}(B)$ and
$U,V, \in \mathcal{L}(F)$, then we have
\begin{eqnarray*}
Ric(X,Y)&=& Ric^B (X,Y) - s[h^{k}_{B}(X,Y) + X(k)Y(k)];\\
Ric(X,V)&=& (s-1)XV(k).
\end{eqnarray*}
\end{pro}

\noindent
A twisted product manifold $B\times_b F$ is called \textit{mixed Ricci-flat} if
$Ric(X,U)=0$ for all $X\mathcal{L}(B)$ and $U\in \mathcal{L}(F)$.

\begin{pro}[~\cite{KS}]
Let $M=B\times_f F$ be a twisted product manifold with a twisting 
function $f$. Then for $X,Y \in \mathcal{L}(B)$ and 
$U,V \in \mathcal{L}(F)$, we have
\begin{eqnarray*}
C(X,Y)V &=& \Big(\frac{1-s}{n-2}\Big)[XV(k)Y - YV(k)X]\\
C(V,W)X &=& \Big(\frac{r-1}{n-2}\Big)[XV(k)W - XW(k)V].
\end{eqnarray*} 
\end{pro}

\noindent
We say that $B\times_f F$ is \textit{mixed Weyl conformal flat} if $C(X,V)=0$
for all $X\in \mathcal{L}(B)$ and $V \in \mathcal{L}(F)$. Moreover, $F$ is
\textit{Weyl conformal-flat} along $B$ if $C(X,Y)=0$, and $B$ is \textit{Weyl
conformal-flat} along $F$ if $C(U,V)=0$ for all $X,Y\mathcal{L}(B)$ and
$U,V \in \mathcal{L}(F)$.\\

\noindent
In ~\cite{FGKU}, Fernandez-Lopez, Garcia-Rio, Kupeli and Unal gave
characterization of a tiwsted product manifold to be a warped product
manifold using the Ricci tensor of the manifold. Similar characterization
were given by Kazan and Sahin using the Weyl conformal curvature tensor
and the Weyl projective curvature tensor in ~\cite{KS}.

\section{Dualistic structures on twisted product manifolds}

\noindent
Let $\bar{X},\bar{Y},\bar{Z} \in \Gamma(TB)$ and $X,Y,Z \in \mathcal{L}(B)$ 
be their corresponding horizontal lifts respectively, we put:
\begin{eqnarray*}
\pi_* (D_X Y) = {}^B\nabla_{\bar{X}} \bar{Y} \quad \mbox{and} \quad
\pi_* (D^{*}_{X} Y) = {}^B\nabla^{*}_{\bar{X}} \bar{Y}.
\end{eqnarray*}
Since $D$ are $D^*$ are affine connections on $B\times F$ and $\pi$ 
is a projection of $B\times F$ on $B$ then ${}^B\nabla$ and 
$^{B}\nabla^{*}$ are affine connections on $B$.\\
  
\noindent
Let $\bar{U},\bar{V},\bar{W} \in \Gamma(TF)$ and $U,V,W \in \mathcal{L}_V (F)$ 
their corresponding vertical lifts, we put: 
\begin{eqnarray*}
\sigma_* (D_U V) = {}^F\nabla_{\bar{U}} \bar{V} \quad \mbox{and} \quad
\sigma_* (D^{*}_{U} V) = {}^F\nabla^{*}_{\bar{U}} \bar{V}.
\end{eqnarray*}
Since $D$ are $D^*$ are affine connections on $B\times F$ and $\sigma$ is a 
projection of $B\times F$ on $F$ then ${}^F\nabla$ and ${}^F\nabla^{*}$ are 
affine connections on $F$. 
\begin{pro}
Let $(g,D,D^*)$ be a dualistic structure on a twisted product manifold 
$B \times_f F$. Then the projections induces dualistic structures on the 
base and the fiber manifolds.
\end{pro}

\begin{proof}
From (\ref{e1}) and (\ref{e2}), we have:
\begin{eqnarray*}
\bar{X}\cdot g_B (\bar{Y},\bar{Z})\circ \pi &=& X\cdot gY,Z)\nonumber\\
&=& \Big[g(D_X Y,Z) + g(Y,D^{*}_{X} Z)\Big]\nonumber\\
&=& \Big[g_B(\pi_* (D_X Y),\pi_*(Z))\circ \pi \nonumber\\
&+& g_B(\pi_* (Y),\pi_*(D^{*}_{X} Z))\circ \pi\Big]\nonumber\\
&=& \Big[ g_B ({}^B\nabla_{\bar{X}} \bar{Y},\bar{Z}) 
+ g_B(\bar{Y},{}^B\nabla^{*}_{\bar{X}} \bar{Z})\Big]\circ \pi.
\end{eqnarray*}
Thus

\begin{eqnarray*}
\bar{X}\cdot g_B (\bar{Y},\bar{Z}) = g_B ({}^B\nabla_{\bar{X}} \bar{Y},\bar{Z}) 
+ g_B(\bar{Y},{}^B\nabla^{*}_{\bar{X}} \bar{Z}).
\end{eqnarray*}
Hence ${}^B\nabla$ and ${}^B\nabla^{*}$ are conjugate with respect to $g_B$.\\

\noindent
From (\ref{e1}) and (\ref{e3}), we have:
\begin{eqnarray*}
\bar{U}\cdot g_F (\bar{V},\bar{W})\circ \sigma &=& b^{-2} U\cdot g(V,W)\nonumber\\
&=& b^{-2} \Big[g(D_U V,W) + g(V,D^{*}_{U} W)\Big]\nonumber\\
&=& b^{-2} \Big[b^{2} g_F(\sigma_* (D_U V),\sigma_*(W))\circ \sigma \nonumber\\
&+& b^{2} g_F(\sigma_* (V),\sigma_*(D^{*}_{U} W))\circ \sigma \Big]\nonumber\\
&=& \Big[ g_F ({}^F\nabla_{\bar{U}} \bar{V},\bar{W}) 
+ g_F(\bar{V},{}^F\nabla^{*}_{\bar{U}} \bar{W})\Big]\circ \sigma.
\end{eqnarray*}
Hence ${}^F\nabla$ and ${}^F\nabla^{*}$ are conjugate with respect to $g_F$.
\end{proof}

\noindent
Now, we construct a dualistic structure on the doubly warped product space
from those on  its base and fiber manifolds.

\begin{pro}
Let $(g_B,{}^B\nabla,^{B}\nabla^{*})$ and $(g_F,{}^F\nabla,^{F}\nabla^{*})$ 
be dualistic structures on $B$ and $F$. Then the triple $(g,D,D^*)$ is a 
dualistic structure on $B\times F$.
\end{pro}

\begin{proof}
Let $X,Y,Z \in \mathcal{L}(B)$. We have:
\begin{eqnarray*}
 X\cdot g(Y,Z) &=& \bar{X}\cdot g_B(\bar{Y},\bar{Z})\circ \pi \nonumber\\
&=& \Big[g_B ({}^B\nabla_{\bar{X}} \bar{Y},\bar{Z}) 
+ g_B(\bar{Y},{}^B\nabla^{*}_{\bar{X}} \bar{Z}) \Big]\circ\pi \nonumber\\
&=& \Big[g_B ({}^B\nabla_{\bar{X}} \bar{Y},\bar{Z})\circ \pi 
+ g_B(\bar{Y},{}^B\nabla^{*}_{\bar{X}} \bar{Z})\circ \pi \Big] \nonumber\\
&=&  g_B(\pi_* (D_X Y),\pi_*(Z))\circ \pi 
+ g_B(\pi_* (Y),\pi_*(D^{*}_{X} Z))\circ \pi\nonumber\\
&=& g(D_X Y,Z) + g(Y,D^{*}_{X} Z).
\end{eqnarray*}
Let $U,V,W \in \mathcal{L}(F)$. We have:
\begin{eqnarray*}
 U\cdot g(V,W) &=& b^2 \bar{U}\cdot g_F(\bar{V},\bar{W})\circ \sigma \nonumber\\
&=& b^2 \Big[g_F ({}^F\nabla_{\bar{U}} \bar{V},\bar{W}) 
+ g_F(\bar{V},{}^F\nabla^{*}_{\bar{U}} \bar{W}) \Big]\circ \sigma \nonumber\\
&=& b^2 \Big[g_F ({}^F\nabla_{\bar{U}} \bar{V},\bar{W})\circ \sigma 
+ g_F(\bar{V},{}^F\nabla^{*}_{\bar{U}} \bar{W})\circ \sigma \Big] \nonumber\\
&=& b^2 g_F(\sigma_* (D_U V),\sigma_*(W))\circ \sigma 
+ b^2 g_F(\sigma_* (V),\sigma_*(D^{*}_{U} W))\circ \sigma\nonumber\\
&=& g(D_U V,W) + g(V,D^{*}_{U} W).
\end{eqnarray*}
\end{proof}

We call $(g,D,D^*)$ the dualistic strucure on $B\times F$ induced from 
$(g_B,{}^B\nabla,{}^B\nabla^*)$ on $B$ and $(g_F,{}^F\nabla,{}^F\nabla^*)$. We
have the following result:

\begin{pro}
Let $(g,D,D^*)$ the dualistic strucure on $B\times F$ induced from 
$(g_B,{}^B\nabla, {}^B\nabla^*)$ on $B$ and $(g_F,{}^F\nabla, {}^F\nabla^*)$. If 
the connections ${}^B\nabla, {}^B\nabla^*,{}^F\nabla$ and ${}^F\nabla^*$ are 
symmetric and torsion free, then the induced connections $D$ and $D^*$ are also
symmetric and torsion free.
\end{pro}

\section{Dually flat twisted product manifolds}

Let $(g,D,D^*)$ the dualistic strucure on $B\times F$ induced from the
dualistic structures $(g_B,{}^B\nabla,{}^B\nabla^*)$ and 
$(g_F,{}^F\nabla, {}^F\nabla^*)$ on $B$ and $F$ respectively. By 
Proposition 2.5, we have:

\begin{lem}
Let $R, {}^B R$ and ${}^FR$ be the Riemannian curvature operators 
with respect to $\nabla, {}^B\nabla$ and ${}^F\nabla$ respectively. It holds
\begin{eqnarray*}
R(X,Y)Z &=& {}^BR(X,Y)Z;\\
R(X,Y)U &=& 0;\\
R(X,U)Y &=& \frac{h^{b}_{B}(X,Y)}{b}U;\\
R(U,V)X &=& UX(\log b)V -VX(\log b) U;\\
R(X,U)V &=& [VX(\log b)]U -g(U,V) \Big[\frac{\nabla^{B}_{X} (grad_B b)}{b}
+ grad_F (X(\log b)) \Big];\\
R(U,V)W &=& {}^FR(U,V)W + g(U,W)grad_B (V(\log b))
-g(V,U)grad_B(U(\log b))\\ 
&-& \frac{\vert grad_B b\vert^2}{b^2}[g(V,W)U - g(U,W)V].
\end{eqnarray*}
and let $R^*, {}^B R^*$ and ${}^FR^*$ be the Riemannian curvature operators 
with respect to $\nabla^*, {}^B\nabla^*$ and ${}^F\nabla^*$ respectively.
\begin{eqnarray*}
R^*(X,Y)Z &=& {}^B R^*(X,Y)Z;\\
R^*(X,Y)U &=& 0;\\
R^*(X,U)Y &=& \frac{h^{b}_{B}(X,Y)}{b}U;\\
R^*(U,V)X &=& UX(\log b)V -VX(\log b) U;\\
R^*(X,U)V &=& [VX(\log b)]U -g(U,V) \Big[\frac{\nabla^{B}_{X} (grad_B b)}{b}
+ grad_F (X(\log b)) \Big];\\
R^*(U,V)W &=& {}^FR^*(U,V)W + g(U,W)grad_B (V(\log b))
-g(V,U)grad_B(U(\log b))\\ 
&-& \frac{\vert grad_B b\vert^2}{b^2}[g(V,W)U - g(U,W)V].
\end{eqnarray*}
\end{lem}


\begin{rmk}~\cite{T}
Assume $M$ is a warped product. Then $(B\times_b F, g, D,D^*)$ is a 
dually flat space if and only if $(B,g_B,{}^B\nabla,{}^B\nabla^*)$ is 
also dually and $(F,g_F,{}^F\nabla, {}^F\nabla^*)$ is a Riemannian 
manifold of constant sectional curvature.
\end{rmk}

Now, we can give our main theorem:

\begin{thm}
Let $B\times_b F$ be a twisted product of $(B,g_B)$ and $(F,g_F)$ with
twisting function $b$ and $\dim F >1$. Assume $Ric(X,V)=0$ for 
all $X\in \mathcal{L}(B)$ and $V\in \mathcal{L}(F)$, then $(B\times_b F,g,D,D^*)$ 
is a dually flat space if and only if $(B,g_B,{}^B\nabla,{}^B\nabla^*)$ is 
dually flat and $(F,g_F,{}^F\nabla,{}^F\nabla^*)$ is of constant sectional 
curvature.
\end{thm}

\begin{proof}
Let $X\in \mathcal{L}(B)$ and $V\in \mathcal{L}(F)$, then from Proposition 2.6,
we have:
\begin{eqnarray*}
Ric(X,V) = (1-s)XV(k).
\end{eqnarray*}
If $Ric(X,V)=0$, then it follows that $XV(k)=0$ and $VX(k)=0$ for all
$X\in \mathcal{L}(B)$ and $V\in \mathcal{L}(F)$. Now, $XV(k)=0$ implies
that $V(k)$ only depends on the points of $F$, and likewise, $VX(k)=0$
implies that $X(k)=0$ only depends on the points of $B$. Thus $k$ can
be expressed as a sum of two functions $\alpha$ and $\beta$ which are
defined on $B$ on $F$, respectively, that is, 
$k(p,q)=\alpha(p)+\beta(q)$ for any $(p,q)\in B\times F$. Hence
$b=\exp(\alpha)\exp(\beta)$, that is, $b(p,q)=\delta(p)\gamma(q)$, where
$\delta=\exp(\alpha)$ and $\beta=\exp(\beta)$ for any $(p,q)\in B\times F$.
Thus we can write $g=g_B \oplus \delta^2 g_{\mathcal{F}}$ where
$g_{\mathcal{F}} = \gamma^2 g_F$, that is, the twisted product manifold
$B\times_b F$ can be expressed as a warped product $B\times_{\delta} F$,
where the metric tensor of $F$ is $g_{\mathcal{F}}$ is given above.
Thus Theorem is obvious from Remark 1.
\end{proof}

\begin{thm}
Let $B\times_b F$ be a twisted product of $(B,g_B)$ and $(F,g_F)$ with
twisting function $b$. Assume either $B$ is Weyl 
conformal flat along $F$ or $F$ is Weyl conformal flat along $B$. Then 
$(B\times_b F,g,D,D^*)$ is a dually flat space if and only if 
$(B,g_B,{}^B\nabla,{}^B\nabla^*)$ is dually flat and 
$(F,g_F,{}^F\nabla,{}^F\nabla^*)$ is of constant sectional curvature.
\end{thm}

\begin{proof}
From Proposition 2.7, it follows that $VX(k)=0$ and $XV(k)=0$. The rest
of the proof is similar to the previous theorem. 
\end{proof}

\begin{thm}
Let $B\times_b F$ be a twisted product manifold. Assume
\begin{enumerate}
\item either the conformal Weyl tensor is parallel and 
$H^k (X) \neq -X(k)\nabla k$ with $\dim B \neq 1$; 
\item or $H^k (X)=-X(k)\nabla k$.
\end{enumerate}
Then $(B\times_b F,g,D,D^*)$ is a dually flat space 
if and only if $(B,g_B,{}^B\nabla,{}^B\nabla^*)$ is dually flat and 
$(F,g_F,{}^F\nabla,{}^F\nabla^*)$ is of constant sectional curvature.
\end{thm}

\begin{proof}
From Theorem 3.6 and Theorem 3.7 of ~\cite{KS} and Corollary 2.5.
\end{proof}


\end{document}